\documentclass[a4paper,11pt,oneside]{article}
\makeatletter

\usepackage[utf8]{inputenc} 
\usepackage[T1]{fontenc}
\usepackage[english]{babel}
\usepackage[autolanguage]{numprint} 
\usepackage{textcomp,pifont} 
\usepackage{enumerate} 

\newcommand\squotes[1]{‘#1’} 
\newcommand\dquotes[1]{“#1”} 
\newcommand\quotes\dquotes

\newenvironment{NB}{\smallskip\noindent\ding{42}\quad\it}{\par\smallskip} 

\usepackage{amsmath,amsthm,amssymb,mathtools}

\theoremstyle{definition}
\newtheorem{dfn}{Definition}

\theoremstyle{plain}

\newtheorem{lem}[dfn]{Lemma}
\newtheorem{ppn}[dfn]{Proposition}
\newtheorem{thm}[dfn]{Theorem}
\newtheorem{cor}[dfn]{Corollary}

\theoremstyle{remark}
\newtheorem{rmk}[dfn]{Remark}

\renewcommand\[{\begin{equation}}
\renewcommand\]{\end{equation}}
\renewcommand\epsilon\varepsilon
\renewcommand\phi\varphi
\renewcommand\leq\leqslant
\renewcommand\geq\geqslant
\renewcommand\setminus\smallsetminus
\renewcommand\emptyset\varnothing
\renewcommand\div{\mathbin{/}}
\renewcommand\Pr\Prbb \DeclareMathOperator\Prbb{\mathbb{P}}

\newcommand\mbb\matbb
\newcommand\mbf\mathbf
\newcommand\mcal\mathcal
\newcommand\mfrak\mathfrak
\renewcommand\mit\mathit
\newcommand\mrm\mathrm
\newcommand\mopen\mathopen
\newcommand\mclose\mathclose
\newcommand\ab\allowbreak

\newcommand\NN{\mathbf{N}}

\newcommand\RR{\mathbf{R}}
\newcommand\ZZ{\mathbf{Z}} 

\newcommand\bcdot{\mathord{\pmb{\boldsymbol{\cdot}}}} 
\newcommand\1[1]{\mathbf{1}_{#1}} 
\newcommand\dx[1]{d\mspace{-1mu}\mathord{#1}} 
\newcommand\tsfrac[2]{\bgroup\textstyle\frac{1}{2}\egroup} 
\newcommand\bstackrel[2]{\mathrel{\underset{#1}{#2}}} 

\DeclareMathOperator\sign{sgn} 
\DeclareMathOperator\card{card} 
\DeclareMathOperator\trace{tr} 

\DeclarePairedDelimiter\abs{\lvert}{\rvert}
\DeclarePairedDelimiter\ceil{\lceil}{\rceil}
\DeclarePairedDelimiter\norm{\lVert}{\rVert}

\newcounter{Ca}
\renewcommand\theCa{\mathrm{\alph{Ca}}}
\def\Calabel#1{\@bsphack\protected@write\@auxout{}{\string\newlabel{#1}{{\theCa}{\thepage}}}\@esphack}
\newcommand\setCa[1]{\stepcounter{Ca}\theCa\Calabel{#1}}

\DeclareMathOperator\E{\mathbb{E}} 
\DeclareMathOperator\var{Var} 
\DeclareMathOperator\cov{Cov} 

\newcommand\Hmoh{\eta} 
\newcommand\Past{\mcal{P}}
\newcommand\Future{\mcal{F}} 
\newcommand\Drift{\mathbf{D}} 
\newcommand\Filtration{(\mcal{B}_t)_t} 

\DeclarePairedDelimiter\integers{\mathopen{[\![}}{\mathclose{[\![}} 

\newcommand\psymb{\mathord{+}}
\newcommand\msymb{\mathord{-}}
\newcommand\Psymb{\mathord{+}_|}
\newcommand\Msymb{\mathord{-}_|} 

\title{No simple arbitrage for fractional Brownian motion}
\author{Rémi Peyre\footnote{CNRS \& Institut Élie-Cartan de Lorraine (Univ.\ Lorraine); \texttt{remi.peyre@univ-lorraine.fr}.}}
\date{August 3, 2015}

\makeatother

\begin{document}

\maketitle

\begin{abstract}
We prove the following result: For $(Z_t)_{t \in \RR}$ a fractional Brownian motion with arbitrary Hurst parameter, there does not exist any stopping time $\tau$ adapted to the natural filtration of the increments of $Z$ such that, with positive probability, $\tau$ a local minimum at right of the trajectory of $Z$.
\end{abstract}

\section{Introduction}

\subsection{Context}\label{sec:Context}

In this article, we consider a filtered probability space $(\Omega, \Pr, \mcal{B}, (\mcal{B}_t)_{t \in \RR})$. Notation ‘$\omega$’ will implicitly refer to eventualities of $\Omega$; we will use it from time to time when needing to make the dependency on the random phenomenon perfectly clear. We consider a (bilateral) Brownian motion $(W_t)_{t \in \RR}$ whose increments are adapted to our filtered space, which means, for all $t \in \RR$, for all $u \leq 0$, $(W_{t + u} - W_t)$ is $\mcal{B}_t$-measurable, while for all $v \geq 0$, $(W_{t + v} - W_t)$ is independent from $\mcal{F}_t$.

We fix once for all some arbitrary parameter $H \in (0, 1)$ (so that, in the sequel, \quotes{absolute} constants may actually depend on $H$) such that $H \neq 1/2$; moreover, in all this article, $(H - 1/2)$ may be referred to as $\Hmoh$. Then we consider the fractional Brownian motion (fBm) $(Z_t)_{t \in \RR}$ driven by $W$ with Hurst parameter $H$, which means that
\[\label{ZofdW}
Z_t \coloneqq C_1 \int_{\RR} \bigl((t - s)_+^{\Hmoh} - (-s)_+^{\Hmoh}\bigr) \dx{W}_s
,\]
(with the convention that $0^r = 0\ \forall r \in \RR$), where
\[\label{C1}
C_1 \coloneqq \biggl(\frac{1}{2 H} + \int_0^{\infty} \bigl((1 + s)^\eta - s^\eta\bigr)^2 \dx{s}\biggr)^{-1/2}
.\]
Then, the properties of $Z$ are well known: it is a centred Gaussian process whose increments are adapted to $\Filtration$, with $\var (Z_t - Z_s) = \abs{t - s}^{2 H}$ and $Z_0 = 0$ a.s.; its trajectories are locally $(H - \epsilon)$-Hölder with divergence in $O (\abs{t}^{H + \epsilon})$ at infinity, etc.\ (see \cite[Chap.~2]{Nourdin-book-fBm} by example).

\begin{rmk}
Here the integral in the r-h.s.\ of \eqref{ZofdW} should be seen as a deterministic integral rather than as an Itô integral. Indeed, integrating by parts, one has: (here in the case $t > 0$),
\begin{multline}\label{ZofW.0}
C_1^{-1} Z_t
= \int_{\RR} \bigl((t - s)_+^{\Hmoh} - (-s)_+^{\Hmoh}\bigr) \dx{W}_s \\
= \int_{-\infty}^0 \bigl((t - s)^{\Hmoh} - (-s)^{\Hmoh}\bigr) \dx{W}_s + \int_0^t (t - s)^{\Hmoh} \dx{W}_s \\
= \Bigl[\bigl((t - s)^{\Hmoh} - (-s)^{\Hmoh}\bigr) W_s\Bigr]_{s = -\infty}^0 - \Hmoh \int_{-\infty}^0 \bigl((t - s)^{\Hmoh - 1} - (-s)^{\Hmoh - 1}\bigr) W_s \dx{s} \\
+ \Bigl[(t - s)^{\Hmoh} (W_s-W_t)\Bigr]_{s = 0}^t - \Hmoh \int_0^t (t - s)^{\Hmoh - 1} (W_s - W_t) \dx{s} \\
= t^{\Hmoh} W_t - \Hmoh \int_{-\infty}^t \bigl((t - s)^{\Hmoh - 1} - (-s)_+^{\Hmoh - 1}\bigr) (W_s - \1{s > 0} W_t) \dx{s}
,\end{multline}
where all the computations are licit (with absolutely converging integrals) because of the properties of regularity and slow divergence of the (ordinary) Brownian motion.
\end{rmk}

\begin{rmk}
It has to be stressed that in all this article, actually we are not interested in the values themselves of the processes $W$ and $Z$, but rather in their \emph{increments}. This way, the fact that $W_0, Z_0 = 0$ should be considered as a mere convention, completely unessential though convenient.
\end{rmk}

\subsection{Main result}

Now we turn to defining the central concept of this article, which we call \quotes{arbitrage stopping times}:

\begin{dfn}[Local minimum at right]
For $X\colon \RR \to \RR$ a (deterministic) trajectory, $t \in \RR$, we say that $t$ is a \emph{local minimum at right} (l.m.a.r.) for $X$ when the following holds:
\[
\exists \epsilon > 0 \quad \forall v \in [0, \epsilon] \quad X_{\tau + v} \geq X_\tau
.\]
\end{dfn}

\begin{dfn}[Arbitrage time]
For $(X_t)_{t \in \RR}$ a real random process, for $\tau$ a random time, we say that $\tau$ is \emph{arbitrage} for $Z$ when there is positive probability that $\tau (\omega)$ is a local minimum at right for $Z (\omega)$.
\end{dfn}

In this article, our goal will be to prove the following

\begin{thm}\label{thm:main}
In the context of \S~\ref{sec:Context}, there does not exist any stopping time adapted to $\Filtration$ which would be arbitrage for $Z$.
\end{thm}

\begin{rmk}
As the increments of $Z$ are adapted to the filtration $\Filtration$, obviously in Theorem~\ref{thm:main} we may replace that filtration by the filtration generated by the increments of $Z$.
\end{rmk}

\begin{rmk}
In this article we are only considering the case $H \neq 1/2$, but Theorem~\ref{thm:main} is trivially valid for $H = 1/2$ too, since then the fBm $Z$ is nothing but the ordinary Brownian motion $W$ itself, for which the result follows immediately from the Markov property and the local properties of oBm.
\end{rmk}

Stated informally, Theorem~\ref{thm:main} means that, if you are discovering the trajectory of a fractional Brownian motion along time, you cannot find a time at which you might foresee that the trajectory would go on upwards. So, this is a kind of very weak \quotes{martingale} property for the fBm, showing that the existence of correlations for it does not allow you to make anything yet.

The motivation for Theorem~\ref{thm:main} comes from the article \cite{Bender2012} by C.~Bender, where it is explained that this result would be incompatible with the possibility, for a financial random process undergoing a fractional Brownian motion (or rather an exponential fBm), that it had an opportunity of so-called \quotes{simple arbitrage}. [\emph{cf.} \cite[Prop.~3.3]{Bender2012}]. So, our theorem shows that making an arbitrage on a fBm is necessarily \quotes{complicated}.

\subsection{Outline of the proof}

The sequel of this paper is devoted to proving Theorem~\ref{thm:main}. In \S~\ref{sec:Conditional_future}, we will see how one can get rid of the notion of stopping time to get Theorem~\ref{thm:main} back to a result on the trajectories of the fractional Brownian motion. In \S~\ref{sec:LIL-fBm}, we will make the needed result on fBm's trajectories more precise, by establishing a law of iterated logarithm for some variant of the fBm. Next, an issue will be that we have to control the probability of an event being a union over a continuous infinity of $t$'s: that issue will be handled by \S~\ref{sec:Pathwise_control}, in which we will use regularity estimates on the fBm to get our continuous union back to a finite union. Finally, after all these simplifications it will only remain to prove some estimates on Gaussian vectors, which will be the work of \S~\ref{sec:Controlling_Gaussian_vector}.

Some technical results will be postponed to appendices. In particular, in Appendix~\ref{sec:Kuv} we will compute the precise expression of the \quotes{drift operator} appearing in Lemma~\ref{lem:Drift+Levy} describing what the law of the fBm becomes when you condition it by a stopping time: this formula, though not actually required to prove our main result, looks indeed intrinsically worthy to be written down to my eyes. Also, in Appendix~\ref{sec:thick} we will investigate some basic properties of \quotes{thick} subsets of $\NN$; and in Appendices~\ref{sec:supGauss} and~\ref{sec:alm-diag-mat} we will prove two lemmas on resp.\ the supremum of Gaussian processes and the inverse of nearly diagonal matrices.

\section{Conditional future of the fractional Brownian motion}\label{sec:Conditional_future}

\subsection{Preliminary definitions}

To begin with, it will be convenient to set some notation for certain sets of trajectories:

\begin{dfn}[Sets $\Past$ and $\Future$]\label{dfn:P,F}\strut\par
\begin{itemize}
\item
We denote by $\Past$ [like \quotes{past}] the set of the (deterministic) paths $(X_u)_{u \leq 0}$ such that:
\begin{enumerate}
\item\label{itm:P-zero}
$X_0 = 0$;
\item\label{itm:P-local}
$X$ is locally $(H - \epsilon)$-Hölder for all $\epsilon>0$;
\item\label{itm:P-infty}
$X_u \stackrel{u \to -\infty}{=} O (\abs{u}^{H + \epsilon})$ for all $\epsilon > 0$.
\end{enumerate}
\item Similarly, we denote by $\Future$ [like \quotes{future}] the space of the paths $(X_v)_{v \geq 0}$ satisfying the analogues of conditions \ref{itm:P-zero}–\ref{itm:P-infty} for non-negative times.
\end{itemize}
\end{dfn}

\begin{rmk}
With the notation of Definition \ref{dfn:P,F}, one has almost-surely that, for all $t \in \RR$, $\bigl(Z (\omega)_{t + u} - Z (\omega)_t\bigr)_{u \leq 0} \in \Past$ and $\bigl(Z (\omega)_{t + v} - Z (\omega)_t\bigr)_{v \geq 0} \in \Future$. [\emph{cf.}\ \cite[Prop.~1.6 \& Prop.~2.2.3]{Nourdin-book-fBm}].
\end{rmk}

We also define a certain \quotes{drift operator}:

\begin{dfn}[\quotes{Drift operator} $\Drift$]
Let $\Drift \colon \Past \to \Future$ be the linear operator such that, for $X \in \Past$:
\[\label{DriftAsIntegral.1}
\bigl(\Drift X\bigr)_v \coloneqq \int_{-\infty}^0 K (u, v) X_u \dx{u}
,\]
where
\begin{multline}\label{Kuv}
K (u, v) \coloneqq
\frac{\Hmoh}{\Pi (\Hmoh) \Pi (-\Hmoh)} \times \Bigl\{ \\
\Hmoh \int_{-\infty}^0 \bigl(\1{s > u} \xi_{\Hmoh - 1} (s - u, v) \xi_{-\Hmoh - 1}(-s, s - u) - \xi_{\Hmoh - 1} (-u, v) \xi_{-\Hmoh - 1} (-s, -u)\bigr) \dx{s} \\
- v (v - u)^{\Hmoh - 1} (-u)^{-\Hmoh - 1}\Bigr\}
,\end{multline}
where $\Pi (\bcdot)$ is Euler's pi function extrapolating the factorial, and where we denote, for $r \in \RR$, $a, b > 0$:
\[
\xi_r (a, b) \coloneqq (a + b)^r - a^r
.\]
\end{dfn}

\begin{rmk}
Note that, since $H \in (0, 1)$, the integrals in \eqref{Kuv} and \eqref{DriftAsIntegral.1} do converge (absolutely) indeed, and $\Drift$ is well-defined on the whole $\Past$ with values in $\Future$.
\end{rmk}

\begin{rmk}
The equations \eqref{DriftAsIntegral.1}-\eqref{Kuv} defining $\Drift$, though interesting as such, shall not play an essential role in this article. What is really important to have in mind is the moral \emph{meaning} of this operator: actually $\Drift$ was defined so that, informally,
\[\label{Drift}
(\Drift X)_v = \E \bigl(Z_v\big|\ (Z_u)_{u \leq 0} = (X_u)_{u \leq 0}\bigr)
.\]
The formal meaning of \eqref{Drift} will be made clear by Lemma~\ref{lem:Drift+Levy} below.
\end{rmk}

Finally we define a process called the \quotes{Lévy fractional Brownian motion}, which is a kind of unilateral version of the \quotes{regular} fBm:

\begin{dfn}[Lévy fBm]
If $(W_v)_{v \geq 0}$ is a (unilateral) ordinary Brownian motion, then the process $(Y_v)_{v \geq 0}$ defined by
\[\label{LevyfBm}
Y_v \coloneqq C_1 \int_0^v (v - s)^{\Hmoh} \dx{W}_s
\]
(interpreted \emph{via} the same integration by parts trick as in \eqref{ZofW.0}) (and where we recall that $C_1$ is defined by \eqref{C1}) is called a \emph{Lévy fractional Brownian motion} (with Hurst parameter $H$)—or, more accurately, the law of this process (which \eqref{LevyfBm} defines without ambiguity) is called \quotes{the law of the Lévy fBm}.
\end{dfn}

\begin{rmk}
From the regularity properties of the oBm, it is easy to check that the trajectories of the Lévy fBm lie in $\Future$ a.s..
\end{rmk}

\subsection{Conditioning lemma}\label{sec:Drift+Levy}

Now we can state the key lemma of this section:

\begin{lem}\label{lem:Drift+Levy}
In the context of \S~\ref{sec:Context}, for $\tau$ a stopping time, $\bigl((Z_{\tau + v} - Z_\tau)_{v \geq 0} - \Drift ((Z_{\tau + u} - Z_\tau)_{u \leq 0})\bigr)$ is independent of $\mcal{B}_\tau$, and its law is the Lévy fBm.
\end{lem}

\begin{rmk}
In other words, Lemma~\ref{lem:Drift+Levy} states that, conditionally to $\mcal{B}_\tau$ (or, morally, knowing the past trajectory of $Z$ until $\tau$), the law of the future trajectory of $Z$ is equal to a \quotes{deterministic} \emph{drift term} $\Drift((Z_{\tau + u} - Z_\tau)_{u \leq 0})$ plus a \quotes{random} \emph{noise term} being a Lévy fBm.
\end{rmk}

\begin{proof}[\proofname\ of Lemma~\ref{lem:Drift+Levy}]
As the increments of $W$ are adapted to $\Filtration$, conditionally to $\mcal{B}_\tau$, the past trajectory $(W_{\tau + u} - W_\tau)_{u \leq 0}$ of (the increments of) $W$ is deterministic, while its future trajectory $(W_{\tau + v} - W_\tau)_{v \geq 0}$ still has the unconditioned law of a standard oBm. Therefore, for $t \geq 0$, we split
\begin{multline}
Z_{\tau + t} - Z_{\tau} \bstackrel{\eqref{ZofdW}}{=} C_1 \int_{s \in \RR} \bigl((\tau + t - s)_+^{\Hmoh} - (\tau - s)_+^{\Hmoh}\bigr) \dx{W}_s \\
\bstackrel{s \leftarrow s - \tau}{=} C_1 \int_{s \in \RR} \bigl((t - s)_+^{\Hmoh} - (-s)_+^{\Hmoh}\bigr) \dx{W}_{\tau + s} \\
= C_1 \int_{u = -\infty}^0 \bigl((t - u)^{\Hmoh} - (-u)^{\Hmoh}\bigr) \dx{W}_{\tau + u} + C_1 \int_{v = 0}^{t} (t - v)^{\Hmoh} \dx{W}_{\tau + v}
,\end{multline}
in which the first term is deterministic and given by some function of $(W_{\tau + u} - W_{\tau})_{u \leq 0}$, while the second term (seen as a trajectory indexed by $t$) has the law of the Lévy fBm indeed.

To end the proof, it remains to show that the aforementioned first term (seen as a trajectory indexed by $t$) is equal to $\Drift((Z_{\tau + u} - Z_\tau)_{u \leq 0})$ indeed. Since this point is actually not needed to prove our main result, we will postpone it to Appendix~\ref{sec:Kuv}.
\end{proof}

\subsection{Reformulation of the main theorem}

Thanks to Lemma~\ref{lem:Drift+Levy}, we will be able to get a sufficient condition for Theorem~\ref{thm:main} in which there are no stopping times any more. For this we need first an \emph{ad hoc} definition:

\begin{dfn}[Arbitrage path]
We say that a deterministic path $(X_u)_{u \leq 0} \in \Past$ is \emph{arbitrage}, and we denote \quotes{$X \in \mcal{A}$}, when, for $Y$ a Lévy fBm:
\[
\Pr \bigl(\text{$0$ is a local minimum at right for $(\Drift X + Y (\omega))\bigr)$} > 0
.\]
\end{dfn}

Now, Theorem~\ref{thm:main} will be a consequence of the following

\begin{ppn}\label{ppn:toprove-recoded}
In the context of \S~\ref{sec:Context}:
\[\label{toprove-recoded}
\Pr \bigl(\exists t \in \RR \quad (Z (\omega)_{t + u} - Z (\omega)_t)_{u \leq 0} \in \mcal{A}\bigr) = 0
.\]
\end{ppn}

\begin{proof}[\proofname\ of Theorem~\ref{thm:main} from Proposition~\ref{ppn:toprove-recoded}]
Let $\tau$ be any stopping time. Then, writing the law of total probability w.r.t.\ the $\sigma$-algebra $\mcal{B}_\tau$:
\begin{multline}
\Pr (\text{$\tau (\omega)$ is a local minimum at right for $Z (\omega)$}) \\
= \E \bigl(\Pr \bigl(\text{$\tau$ is a l.m.a.r.\ for $Z (\omega')$}\big|\ \mcal{B}_{\tau}\bigr) (\omega)\bigr) \footnotemark \\
= \E \bigl(\Pr \bigl(\text{$0$ is a l.m.a.r.\ for $(Z (\omega')_{\tau + v} - Z (\omega')_\tau)_{v \geq 0}$}\big|\ \mcal{B}_{\tau}\bigr) (\omega)\bigr) \\
\bstackrel{\text{Lem.~\ref{lem:Drift+Levy}}}{=} \E \bigl(\Pr \bigl(\text{$0$ is a l.m.a.r.\ for $\bigl(\Drift ((Z_{\tau + u} - Z_\tau)_{u \leq 0})$} + Y (\omega')\bigr)\bigr) (\omega)\bigr) \\
= \E \bigl(\text{$0$ whenever $((Z (\omega)_{\tau + u} - Z (\omega)_\tau)_{u \leq 0}) \notin \mcal{A}$}\bigr)
\bstackrel{\text{Prop.~\ref{ppn:toprove-recoded}}}{=} 0
,\end{multline}
\footnotetext{In this computation I am using notation \squotes{$\omega'$} when referring to an eventuality for the conditional law $\Pr (\bcdot|\ \mcal{B}_\tau)$, while \squotes{$\omega$} is reserved to eventualities for the unconditioned law.}
so that $\tau$ is not arbitrage.
\end{proof}

So, in the sequel, our new goal will be to prove Proposition~\ref{ppn:toprove-recoded}.

\section{Local behaviour of fBm's trajectories}\label{sec:LIL-fBm}

\subsection{A law of the iterated logarithm for the Lévy fBm}

\begin{NB}
In all this article, we denote $\integers{n} \coloneqq \NN \cap [0, n) = \{0,\ab 1,\ab 2, \ldots,\ab n - 1\}$. A subset $\mcal{I} \subset \NN$ will be said to be \emph{thick} when it has positive upper asymptotic density:
\[
\limsup_{n\to\infty} \frac{\abs{\mcal{I} \cap \integers{n}}}{n} > 0
.\]
A few basic results on thick subsets of $\NN$ are gathered in Appendix \ref{sec:thick}.
\end{NB}

The first main result of this section is the following
\begin{lem}\label{lem:LIL}
Let $r \in (0, 1)$ and let $\mcal{I}$ be a thick subset of $\NN$; then, for $Y$ a Lévy fBm:
\[\label{LIL-LfBm}
\liminf_{\substack{i \in \mcal{I}\\ i \to \infty}} \frac{Y_{r^i}}{(\log i)^{1/2} r^{H i}} = -H^{-1/2} \quad \text{a.s.}
.\]
\end{lem}

\begin{rmk}
Actually only the \squotes{$\leq$} sense of \eqref{LIL-LfBm} (which is the harder one) will be needed in this article.
\end{rmk}

\begin{proof}[\proofname\ of Lemma \ref{lem:LIL}]
It will be convenient in this proof to assume that $Y$ is driven by some oBm $W$ according to \eqref{LevyfBm}. Then, for $v \geq 0$, let us define
\[
\tilde{Y}_v \coloneqq \int_{r v}^v (v - s)^{\Hmoh} \dx{W}_s
,\]
resp.
\[
Y'_v \coloneqq Y_v - \tilde{Y}_v = \int_0^{r v} (v - s)^{\Hmoh} \dx{W}_s
.\]

First let us study the $\tilde{Y}_{r^i}$'s. Obviously these random variables are independent, with $Y_{r^i} \div r^{H i} \sim \mcal{N} \bigl(0, (1 - r)^{2 H} \div 2 H\bigr)\ab\ \forall i$. Now, using that $\Pr \bigl(\mcal{N} (0, 1) \leq -x\bigr) \geq e^{-x^2 \div 2} \div 2 \sqrt{2 \pi} x$ for $x \geq 1$,%
\footnote{This is because of convexity of the density $y \mapsto \phi (y) \coloneqq e^{-y^2 \div 2} \div \sqrt{2 \pi}$ on $(-\infty, -1]$: from this property you deduce that $\int_{-\infty}^{-x} \phi (y) \dx{y} \geq \phi (-x)^2 \div 2 \phi' (-x) = \phi (-x) \div 2 x$.}
we get that for $i$ large enough: (having fixed some arbitrary small $\epsilon \in (0, 1)$),
\begin{multline}
\Pr \bigl(\tilde{Y}_{r^i} \div (\log i)^{1/2} r^{H i} \leq -(1 - \epsilon) (1 - r)^H H^{-1/2}\bigr) \\
= \Pr \bigl(\mcal{N} (0, 1) \leq -(1 - \epsilon) \sqrt{2} (\log i)^{1/2}\bigr) \\
\geq i^{-(1 - \epsilon)^2} \div (1 - \epsilon) 4\sqrt{\pi} (\log i)^{1/2}
\stackrel{i\to\infty}{=} \Omega(i^{-1})
,\end{multline}
where \quotes{$f (i) = \Omega (g (i))$} means that $g (i) = O (f (i))$.
As $\mcal{I}$ is thick, the series $\sum_{i\in\mcal{I}} i^{-1}$ is divergent (\emph{cf.}\ Lemma \ref{lem:harmonicthick} in Appendix~\ref{sec:thick}), thus so is
\[
\sum_{i\in\mcal{I}} \Pr \biggl(\frac{\tilde{Y}_{r^i}}{(\log i)^{1/2} r^{H i}} \leq -(1-\epsilon) (1-r)^H H^{-1/2} \biggr) .
\]
Since the events concerning the different $\tilde{Y}_{r^i}$'s are independent, it follows by the (second) Borel–Cantelli lemma that almost-surely there are infinitely many $i \in \mcal{I}$ for which $\tilde{Y}_{r^i} \div (\log i)^{1/2} r^{H i} \leq -(1 - \epsilon) (1 - r)^H H^{-1/2}$, so that:
\[\label{liminfYtilde}
\liminf_{\substack{i \in \mcal{I}\\ i \to \infty}} \frac{\tilde{Y}_{r^i}}{(\log i)^{1/2} r^{H i}} \leq -(1 - \epsilon) (1 - r)^H H^{-1/2} \quad \text{a.s.}
,\]
in which the factor $(1 - \epsilon)$ may be removed by letting $\epsilon \to 0$.

Now let us handle the $Y'_{r^i}$'s. One has $Y'_{r^i} \div r^{H i} \sim \mcal{N} \bigl(0, 1 - (1 - r)^{2H} \div 2H\bigr)$; therefore, using that $\Pr \bigl(\mcal{N} (0, 1) \geq x\bigr) \leq e^{-x^2 \div 2}$ for all $x$,%
\footnote{This is because $\E (e^{x \mcal{N} (0, 1)}) = e^{x^2 \div 2}$, from which the claimed formula follows by Markov's inequality.}
we get that: (having fixed some arbitrary small $\epsilon > 0$),
\begin{multline}
\Pr \bigl(Y'_{r^i} \div (\log i)^{1/2} r^{H i} \geq (1 + \epsilon) \bigl(1 - (1 -r)^{2 H}\bigr)^{1/2} H^{-1/2}\bigr) \\
= \Pr\bigl(\mcal{N} (0, 1) \geq (1 + \epsilon) \sqrt{2} (\log i)^{1/2}\bigr)
\leq i^{-(1+\epsilon)^2}
.\end{multline}
The series $\sum_{i \in \mcal{I}} i^{-(1 + \epsilon)^2}$ is convergent since $\sum_{i\in\NN} i^{-(1+\epsilon)^2}$ is, thus so is
\[
\sum_{i \in \mcal{I}} \Pr \biggl(\frac{Y'_{r^i}}{(\log i)^{1/2} r^{H i}} \geq (1 + \epsilon) \bigl(1 - (1 - r)^{2H}\bigr)^{1/2} H^{-1/2}\biggr)
.\]
It follows by the (first) Borel–Cantelli lemma that almost-surely there are only finitely many $i$'s for which $\tilde{Y}_{r^i} \div (\log i)^{1/2}r^{H i} \geq (1+\epsilon) (1-(1-r)^{2H})^{1/2} H^{-1/2}$, so that:
\[\label{limsupYprime}
\limsup_{\substack{i \in \mcal{I}\\ i \to \infty}} \frac{Y'_{r^i}}{(\log i)^{1/2} r^{H i}} \leq (1 + \epsilon) \bigl(1-(1-r)^{2H}\bigr)^{1/2} H^{-1/2} \quad \text{a.s.}
,\]
in which the factor $(1+\epsilon)$ may be removed by letting $\epsilon \to 0$.

Summing \eqref{liminfYtilde} and \eqref{limsupYprime}, we get an intermediate result:

\begin{ppn}\label{ppn:2513}
Under the assumptions of Lemma~\ref{lem:LIL}, almost-surely:
\[\label{2513}
\liminf_{i \in \mcal{I}} \frac{Y_{r^i}}{(\log i)^{1/2} r^{H i}} \leq -\lambda(r) ,
\]
where
\[
\lambda(r) \coloneqq \bigl((1 - r)^H - \bigl(1 - (1 - r)^{2 H}\bigr)^{1/2}\bigr) H^{-1/2}
.\]
\end{ppn}

So, now it remains to improve the constant $\lambda (r)$ in \eqref{2513} into $H^{-1/2}$. For this, we begin with observing that  the Lévy fBm is scale invariant with exponent $H$ (by which I mean that for $a \in \RR_+^*$, $(Y_{av} \div a^H)_{v \geq 0}$ is also a Lévy fBm); therefore, Proposition~\ref{ppn:2513} has the following
\begin{cor}\label{cor:0363}
Under the assumptions of Lemma~\ref{lem:LIL}, for $a \in \RR_+^*$, one has almost-surely:
\[
\liminf_{i \in \mcal{I}} \frac{Y_{a r^i}}{a^H (\log i)^{1/2} r^{H i}} \leq -\lambda(r)
.\]
\end{cor}

Now let $k > 1$ be an arbitrary large integer; and take $l \in \integers{k}$ such that $\mcal{J} \coloneqq \{j \in \NN|\ \ab k j + l \in \mcal{I}\}$ is thick—existence of such an $l$ is guaranteed by Lemma~\ref{lem:2179} in Appendix~\ref{sec:thick}. Then one has:
\begin{multline}\label{3923}
\liminf_{i\in\mcal{I}} \frac{Y_{r^i}}{(\log i)^{1/2} r^{H i}} \leq
\liminf_{j\in\mcal{J}} \frac{Y_{r^{k j+l}}} {\bigl(\log (k j+l)\bigr)^{1/2} r^{H(k j+l)}} \\
= \liminf_{j\in\mcal{J}} \frac{Y_{r^l(r^k)^j}} {(r^l)^H (\log j)^{1/2} (r^k)^{H j}}
,\end{multline}
where in the last equality we used that $\bigl(\log(k j + l)\bigr)^{1/2} \stackrel{j \to \infty}{\sim} (\log j)^{1/2}$. But, applying Corollary \ref{cor:0363} with $\text{\squotes{$r$}} = r^k$, $\text{\squotes{$a$}} = r^l$ and $\text{\squotes{$\mcal{I}$}} = \mcal{J}$, the r-h.s.\ of \eqref{3923} is bounded above by $-\lambda(r^k)$; so,
\[
\liminf_{i \in \mcal{I}} \frac{Y_{r^i}}{(\log i)^{1/2} r^{H i}} \leq -\lambda(r^k) .
\]
Letting $k$ tend to infinity, $\lambda (r^k)$ tends to $H^{-1/2}$, which proves the \squotes{$\leq$} sense of \eqref{LIL-LfBm}.

For the \squotes{$\geq$} sense, it is the same reasoning as for deriving \eqref{limsupYprime}, just replacing \quotes{$Y'$} by \quotes{$-Y$} and \quotes{$(1 - (1 - r)^{2 H}) \div 2 H$} by \quotes{$1 \div 2 H$}.
\end{proof}

\subsection{Arbitrage condition as a limit}

For all the sequel of this article, we fix some $r \in (0, 1)$ small enough (in a sense to be made precise later); we also fix arbitrarily two parameters $\alpha \in (0, 1)$ and $p \in (0, 1)$. Then we define, for all $n>0$:
\[\label{dfnAn}
\mcal{A}_n \coloneqq \bigl\{(X_u)_{u \leq 0} \in \Past\big|\ \card \{i\in\integers{n}|\ (\Drift X)_{r^i} \geq \alpha H^{-1/2} (\log i)_+^{1/2} r^{H i}\} \geq p n \bigr\}
.\]

Then we have the following connection between $\mcal{A}$ and the $\mcal{A}_n$'s:
\begin{lem}\label{ppn:A<An}
\[
\mcal{A} \subset \liminf_{n \in \NN} \mcal{A}_n .
\]
\end{lem}

\begin{proof}
We prove the contrapositive inclusion. Let $(X_u)_{u \leq 0} \in \Past$ be such that $X \notin \liminf\{\mcal{A}_n\}$, that is, the set $\{n|\ X \notin \mcal{A}_n\}$ is unbounded; and set
\[
\mcal{I} \coloneqq \{i \in \NN|\ (\Drift X)_{r^i} < \alpha H^{-1/2} (\log i)_+^{1/2} r^{H i}\}
.\]
One has by definition that $\abs{\mcal{I} \cap \integers{n}} \div n \geq 1 - p$ for all $n$ such that $X \notin \mcal{A}_n$; as these $n$ are unbounded and $1 - p > 0$, it follows that $\mcal{I}$ is thick.
Therefore, Lemma \ref{lem:LIL} gives that for almost-all Lévy fBm $Y (\omega)$, one has that
\[\label{3332}
\liminf_{i \in \mcal{I}} \frac{Y (\omega)_{r^i}}{(\log i)^{1/2} r^{H i}} \leq -H^{-1/2}
.\]
On the other hand, the definition of $\mcal{I}$ obviously implies that
\[\label{3333}
\liminf_{i \in \mcal{I}} \frac{(\Drift X)_{r^i}}{r^{H i} (\log i)^{1/2}} \leq \alpha H^{-1/2}
.\]
Summing \eqref{3332} and \eqref{3333}, it follows that almost-surely:
\[\label{2859}
\liminf_{i\in\mcal{I}} \frac{(\Drift X + Y (\omega))_{r^i}}{(\log i)^{1/2} r^{H i}} \leq -(1 - \alpha) H^{-1/2} < 0
.\]
But \eqref{2859} implies that $(\Drift X + Y (\omega))_{r^i}$ is negative for values of $r^i$ arbitrarily close to $0$, so that $0$ is almost-surely not a local minimum at right for $(\Drift X + Y (\omega))$; therefore $X \notin \mcal{A}$, which is what we wanted.
\end{proof}

\subsection{Second reformulation of the main theorem}

Thanks to the work of this section, we are now able to show that the following result will be a sufficient condition for Proposition~\ref{ppn:toprove-recoded}:

\begin{ppn}\label{ppn:toprove-01An}
In the context of \S~\ref{sec:Context}:
\[\label{toprove-01An}
\Pr \bigl(\exists t \in [0, 1] \quad (Z (\omega)_{t + u} - Z (\omega)_t)_{u \leq 0} \in \mcal{A}_n\bigr) \stackrel{n \to \infty}{\to} 0
.\]
\end{ppn}

\begin{proof}[\proofname\ of Proposition~\ref{ppn:toprove-recoded} from Proposition~\ref{ppn:toprove-01An}]
Proposition \ref{ppn:A<An} implies that
\begin{multline}
\{\omega \in \Omega|\ \exists t \in [0, 1] \quad (Z (\omega)_{t + u} - Z (\omega)_t)_{u \leq 0} \in \mcal{A}\} \\
\subset \liminf_{n \to \infty} \{\omega|\ \exists t \in [0, 1] \quad (Z (\omega)_{t + u} - Z (\omega)_t)_{u \leq 0} \in \mcal{A}_n\}
;\end{multline}
therefore, by the (first) Borel–Cantelli lemma, Proposition~\ref{ppn:toprove-01An} yields that
\[\label{toprove-01}
\Pr \bigl(\exists t \in [0, 1] \quad (Z (\omega)_{t + u} - Z (\omega)_t)_{u \leq 0} \in \mcal{A}\bigr) = 0
.\]
But, since the increments of the fractional Brownian motion are stationary, in \eqref{toprove-01} we may replace $[0, 1]$ by $[n, n + 1]$ for all $n \in \ZZ$; and then, by countable union, we get the wished result \eqref{toprove-recoded}.
\end{proof}

So, in the sequel, our new goal will be to prove Proposition~\ref{ppn:toprove-01An}.

\section{Pathwise control via pointwise control}\label{sec:Pathwise_control}

\subsection{A regularity result}

One of our issues to prove Proposition~\ref{ppn:toprove-01An} is that we have to bound the probability of an event defined as a union for uncountably infinitely many $t$'s. To overcome this issue, we will need a tool to \quotes{get rid of the trajectorial aspects} of the problem: this is the work of this section.

First, we need a little notation:
\begin{dfn}[Processes $\hat{\Gamma}_i$ and variables $\Gamma_i$]
Within the context of \S~\ref{sec:Context}, for $i \in \NN$, we define the following random process (indexed by $t \in \RR$):
\[\label{dfnhatGam}
\hat{\Gamma}_i(\omega)_t \coloneqq \frac{\Drift \bigl((Z (\omega)_{t + u} - Z (\omega)_t)_{u \leq 0}\bigr) (r^i)}{r^{H i}}
.\]
We also define the following random variable:
\[\label{dfnGam}
\Gamma_i (\omega) \coloneqq \hat{\Gamma}_i (\omega)_0 = \frac{\Drift ((Z (\omega)_u)_{u \leq 0})(r^i)}{r^{H i}}
.\]
\end{dfn}

Then, the main result of this section is the following
\begin{lem}\label{lem:regGamhat}
In the context of this section, there exist absolute%
\footnote{Remember that in this article, \quotes{absolute} constants may actually depend on $H$.}
constants $C_{\setCa{a}} > 0,\ab C_{\setCa{b}} < \infty$ (whose exact expressions do not matter) such that for all $i \in \NN$, for all $T > 0$:
\[\label{regGamhat}
\Pr \bigl(\exists t \in [0, T] \quad \abs{(\hat{\Gamma}_i)_t - \Gamma_i} \geq 1\bigr) \leq C_{\ref{b}} \exp \bigl(-C_{\ref{a}} (r^i \div T)^{2 H \wedge 1}\bigr)
.\]
\end{lem}

\begin{proof}
First, since $Z$ is scale invariant with exponent $H$ (and operator $\Drift$ preserves that scale invariance), it will be enough to prove Lemma~\ref{lem:regGamhat} for $i = 0$; so we will only handle that case. Then the subscript $i$ becomes useless, so we remove it in our notation.

Because of the characterization \eqref{Drift} of $\Drift$, $\hat{\Gamma}_t$ may be written as a function of $W$:
\[\label{GamhatofdW}
\hat{\Gamma}_t = \int_{-\infty}^t \bigl((t + 1 - s)^{\Hmoh} - (t - s)^{\Hmoh}\bigr) \dx{W_s}
.\]
That shows that $\hat{\Gamma}$ is a stationary centred Gaussian random process, with
\begin{multline}\label{Var(Gamhatt-Gamhat0)}
\var \bigl(\hat{\Gamma}_t -\hat{\Gamma}_0\bigr)
= \int_{\RR} \bigl(\1{s \leq t} (t + 1 - s)^{\Hmoh} - \1{s \leq 0} (1 - s)^{\Hmoh} - (t - s)_+^{\Hmoh} + (-s)_+^{\Hmoh}\bigr)^2 \dx{s} \\
\stackrel{t \to 0}{=} O(t^{2 H \wedge 1})
.\end{multline}

To go further, we need the following lemma, whose proof is postponed to Appendix~\ref{sec:supGauss}:
\begin{lem}\label{lem:regGaussian}
Let $(X_t)_{t \in [0, 1]}$ be a centred Gaussian process such that $X_0 = 0$ a.s.\ and
\[
\forall t, s \in [0, 1] \quad \var (X_t - X_s) \leq \abs{t - s}^{2 \theta}
\]
for some $\theta \in (0, 1]$. Then it is known that, by the Kolmogorov continuity theorem, $X$ has a continuous version. The present lemma states that, for this continuous version, the random variable $\norm{X} (\omega) \coloneqq \sup_{t \in [0, 1]} \abs{X (\omega)_t}$ is sub-Gaussian with absolute constants, i.e.\ there exist constants $C_{\setCa{c}} (\theta) > 0,\ab C_{\setCa{d}} (\theta) < \infty$ such that
\[\label{regGaussian}
\forall x \geq 0 \quad \Pr (\norm{X} \geq x) \leq C_{\ref{d}} \exp (-C_{\ref{c}} x^2)
.\]
\end{lem}

We apply Lemma~\ref{lem:regGaussian} in the following way. From \eqref{Var(Gamhatt-Gamhat0)}, one has $\var \bigl(\hat{\Gamma}_t -\hat{\Gamma}_s\bigr) \leq C_{\setCa{e}} \abs{t - s}^{2 H \wedge 1}$ for all $t, s \in [0, 1]$, for some $C_{\ref{e}} < \infty$. Therefore, for $T \leq 1$, the random process
\[
X (\omega)_t \coloneqq (C_{\ref{e}} T^{2 H \wedge 1})^{-1/2} (\hat{\Gamma} (\omega)_{t T} - \Gamma (\omega))
\]
satisfies the assumptions of Lemma~\ref{lem:regGaussian} with $\text{\squotes{$\theta$}} = H \wedge 1/2$, so that \eqref{regGaussian} yields:
\[\label{0028}
\Pr \bigl(\exists t \in [0, T] \quad \abs{(\hat{\Gamma}_i)_t - \Gamma_i} \geq x\bigr) \leq C_{\ref{d}} \exp \bigl(-C_{\ref{c}} C_{\ref{e}}^{-1} x^2 \div T^{2 H \wedge 1}\bigr)
.\]
This implies \eqref{regGamhat} for $T \leq 1$, with constants not depending on $T$. On the other hand, up to replacing $C_{\ref{b}}$ by $(e^{C_{\ref{a}}} \vee C_{\ref{b}})$, \eqref{regGamhat} is automatically true for $T > 1$; so the proof of Lemma~\ref{lem:regGamhat} is completed.
\end{proof}

\subsection{Third reformulation of the main result}

Now, we will see how Proposition~\ref{lem:regGamhat} allows one to find an easier sufficient condition for Proposition~\ref{ppn:toprove-01An}. First of all, we have to introduce a little notation: in this section, we fix some arbitrary $\alpha' \in (0, \alpha),\ab p' \in (0, p)$, and we define $\mcal{A}'_n$ by the variant of Equation~\eqref{dfnAn} in which $\alpha$ and $p$ are replaced by resp.\ $\alpha'$ and $p'$; also, we fix some arbitrary $\tilde{r} \in (0, r)$, and we set
\[
T_n \coloneqq \tilde{r}^n
.\]

Now, we introduce the following events of $\Omega$:

\begin{dfn}[Events $A_n$, $A'_n$, $\bar{A}_n$, $\bar{A}_n^k$ and $\bar{A}_n^*$]\label{dfn:Aevents}
\begin{align}
\label{dfnrmAn}
A_n &\coloneqq \{\omega|\ (Z (\omega)_u)_{u \leq 0} \in \mcal{A}_n \} ;\\
A_n' &\coloneqq \{\omega|\ (Z (\omega)_u)_{u \leq 0} \in \mcal{A}'_n\} ;\\
\bar{A}_n &\coloneqq \{\omega|\ \exists t \in [0, T_n] \quad (Z (\omega)_{t + u} - Z (\omega)_t)_{u \leq 0} \in \mcal{A}_n\} ;\\
\bar{A}_n^k &\coloneqq \{\omega|\ \exists t \in [k T_n, (k + 1) T_n] \quad (Z (\omega)_{t + u} - Z (\omega)_t)_{u \leq 0} \in \mcal{A}_n\} ;\\
\bar{A}_n^* &\coloneqq \{\omega|\ \exists t \in [0, 1] \quad (Z (\omega)_{t + u} - Z (\omega)_t)_{u \leq 0} \in \mcal{A}_n\}
.\end{align}
\end{dfn}

Then we claim that

\begin{lem}\label{lem:barAn-An'}
For $n$ large enough,
\[\label{barAn-An'}
\omega \in \bar{A}_n \setminus A_n' \quad \Rightarrow \quad \exists i \in \integers{n} \quad \exists t \in [0, T_n] \quad \abs{\hat{\Gamma}_i (\omega)_t - \Gamma_i (\omega)} \geq 1
.\]
\end{lem}
\begin{proof}
Assume that $\omega \in \bar{A}_n \setminus A_n'$. Then the fact that $\omega \in \bar{A}_n$ means that there exists some $t \in [0, T_n]$ such that $(Z_{t + u} - Z_t)_{u \leq 0} \in \mcal{A}_n$. For such a $t$, going back to the definitions \eqref{dfnAn} and \eqref{dfnhatGam} of $\mcal{A}_n$ and $\hat{\Gamma}_i$, this means that
\[
\card \{i \in \integers{n}|\ (\hat{\Gamma}_i)_t \geq \alpha H^{-1/2} (\log i)_+^{1/2}\} \geq p n
.\]
Similarly, the fact that $\omega\notin A_n'$ means that
\[
\card \{i \in \integers{n}|\ \Gamma_i \geq \alpha' H^{-1/2} (\log i)_+^{1/2}\} < p' n
.\]
Therefore, there exist at least $(p - p') n$ indices \squotes{$i$} such that $(\hat{\Gamma}_i)_t \geq \alpha H^{-1/2} \* (\log i)_+^{1/2}$ while (for the same $i$) $\Gamma_i < \alpha' H^{-1/2} \* (\log i)_+^{1/2}$. Necessarily one these indices is $\geq (p - p') n - 1$; thus, for such an $i$, one has:
\[\label{1702}
(\hat{\Gamma}_i)_t - \Gamma_i \geq (\alpha - \alpha') H^{-1/2} \bigl(\log \bigl((p - p') n - 1\bigr) \bigr)^{1/2}
.\]
But, provided $n \geq (e^{H \div (\alpha - \alpha')^2} + 1) \div (p - p')$, the r-h.s.\ of \eqref{1702} is $\geq 1$; so in the end we have found $i \in \integers{n},\ab t \in [0, T_n]$ such that $\abs{\hat{\Gamma}_i (\omega)_t - \Gamma_i (\omega)} \geq 1$, proving the lemma.
\end{proof}

Combining Lemma~\ref{lem:barAn-An'} with Lemma~\ref{lem:regGamhat}, we get that
\[
\Pr (\bar{A}_n \setminus A_n') \leq \sum_{i = 0}^{n - 1} C_{\ref{b}} \exp \bigl(-C_{\ref{a}} (r^i \div T_n)^{2 H \wedge 1}\bigr)
,\]
in which the right-hand side is obviously bounded by
\[
n C_{\ref{b}} \exp \bigl(-C_{\ref{a}} (r \div \tilde{r})^{(2 H \wedge 1) n}\bigr),
\]
which shows that $\Pr (\bar{A}_n \setminus A_n')$ decreases superexponentially in $n$ (i.e.\ faster than any exponential).

Now $\bar{A}_n^* \subset \bigcup_{k \in \integers{\ceil{1 \div T_n}}} \bar{A}_n^k$, where $\Pr (\bar{A}^k_n) = \Pr (\bar{A}_n)\ \forall k$ by translation invariance, so it follows that
\[\label{(1/Tn)cases}
\Pr (\bar{A}_n^*) \leq \ceil{1 \div T_n} \Pr (\bar{A}_n) \leq \ceil{\tilde{r}^{-n}} \bigl(\Pr (A_n') + \Pr (\bar{A}_n \setminus A_n')\bigr) =  \ceil{\tilde{r}^{-n}} \Pr (A_n') + o (1)
.\]

Our goal being to prove that $\Pr (\bar{A}_n^*) \to 0$ as $n \to \infty$ (that is just re-writing Proposition~\ref{ppn:toprove-01An} with the notation of this section), it will be sufficient for that to prove the following

\begin{ppn}\label{ppn:P(An') superexponential}
$\Pr (A_n')$ decreases superexponentially in $n$.
\end{ppn}

So, as $A_n'$ corresponds to a condition on a finite-dimensional Gaussian vector, we have managed to get completely rid of the trajectorial aspects of the problem! Now our ultimate goal will be to prove Proposition~\ref{ppn:P(An') superexponential}.

\begin{rmk}
As the \quotes{prime} symbols would be somehow cumbersome, we will drop them  in the sequel, thus actually proving the superexponential decrease of $\Pr (A_n)$. Nevertheless this should not be confusing, as the constraints on $\alpha$ and $p$ (and therefore on $A_n$) are the same as on $\alpha'$ and $p'$ (and therefore on $A_n'$).
\end{rmk}

\section{Final computations: controlling a Gaussian vector}\label{sec:Controlling_Gaussian_vector}

\subsection{Covariance structure}

\begin{NB}
In this section, for \squotes{$X$} a symbol and $I$ a discrete set, \quotes{$\vec{X}_I$} will be a shorthand for \quotes{$(X_i)_{i \in I}$}.
\end{NB}

So, our goal is to prove the superexponential decay of $\Pr (A_n)$, which can be re-written as
\[\label{Pr(A_n)-rewritten}
\Pr (A_n) = \Pr \bigl(\card \{i \in \integers{n}|\ \Gamma_i \geq \alpha H^{-1/2} (\log i)_+^{1/2}\} \geq p n \bigr)
\]
(where $\Gamma_i$ was defined by \eqref{dfnGam}). \eqref{Pr(A_n)-rewritten} obviously implies that
\[
\Pr (A_n) \leq \sum_{\substack{I \subset \integers{n}\\ \abs{I} \geq p n}} \Pr \bigl(\forall i \in I \quad \Gamma_i \geq \alpha H^{-1/2} (\log i)_+^{1/2} \bigr)
.\]
As there are only $2^n$ subsets of $\integers{n}$, to prove that $\Pr(A_n)$ decreases superexponentially it is therefore sufficient to prove that
\[\label{8372}
\sup_{\substack{I \subset\integers{n}\\ \abs{I} \geq p n}} \Pr \bigl(\forall i \in I \quad \Gamma_i \geq \alpha H^{-1/2} (\log i)_+^{1/2}\bigr)
\]
decreases superexponentially.

Now, by \eqref{Drift} one has
\[
\Gamma_i (\omega) = r^{-H i} \int_{-\infty}^{0} \bigl((r^i - s)^\Hmoh - (-s)^\Hmoh\bigr) \dx{W} (\omega)_s
;\]
therefore $\vec{\Gamma}_{\NN}$ is a centred Gaussian vector, with:
\begin{multline}\label{CovGiGj}
\cov (\Gamma_i, \Gamma_j)
= r^{-(i + j) H} \int_{-\infty}^0 \bigl((r^i - s)^\Hmoh - (-s)^\Hmoh\bigr) \bigl((r^j - s)^\Hmoh - (-s)^\Hmoh\bigr) \dx{s} \\
= r^{-\abs{i - j} H} \int_{-\infty}^0 \bigl((r^{\abs{i - j}}-s)^\Hmoh - (-s)^\Hmoh\bigr) \bigl((1 - s)^\Hmoh - (-s)^\Hmoh\bigr) \dx{s} \\
\leq C_{\setCa{f}} r^{(1/2 - \abs{\Hmoh}) \abs{i - j}}
,\end{multline}
for some absolute constant $C_{\ref{f}} < \infty$.
Therefore, provided $r$ was chosen small enough, we have the following control on the covariance matrix of $\vec{\Gamma}_{\NN}$:
\begin{align}
\label{CovGiGj-1a}
\cov (\Gamma_i, \Gamma_j) &= \sigma^2 && \text{for $i = j$;} \\
\label{CovGiGj-1b}
\abs{\cov (\Gamma_i, \Gamma_j)} &\leq \sigma^2\epsilon^{\abs{i - j}} && \text{for $i \neq j$,}
\end{align}
where $\epsilon > 0$ is some small parameter which will be fixed later, and where $\sigma \coloneqq \var (\Gamma)^{1/2} > 0$ (since $H \neq 1/2$).

\subsection{Density estimates}

To exploit \eqref{CovGiGj-1a}-\eqref{CovGiGj-1b}, we need the following lemma (whose proof is postponed to Appendix~\ref{sec:alm-diag-mat}):

\begin{lem}\label{lem:alm-diag-mat}
For $n \in \NN$, $\epsilon > 0$, let $\mbf{A} \eqqcolon (\!(a_{i j})\!)_{i, j \in \integers{n}}$ be a square matrix such that $a_{i i} = 1\ab\ \forall i$ and $\abs{a_{i j}} \leq \epsilon^{\abs{i - j}}\ab\ \forall i \neq j$. Then:
\[\label{det-alm-diag-mat}
\det \mbf{A} \geq \exp(-n \Phi_{\setCa{g}} (\epsilon) \epsilon^2)
,\]
where $\Phi_{\ref{g}} \colon (0, \infty) \to [1, \infty]$ is some absolute function (in particular, not depending on $n$) such that $\Phi_{\ref{g}} (\epsilon) \stackrel{\epsilon\to0}{\to} 1$—we will call such a function a \emph{quasi-one function}. In particular, provided $\epsilon$ is small enough, $\mbf{A}$ is invertible. Then the present lemma asserts moreover that, denoting $\mbf{A}^{-1} \eqqcolon (\!(b_{i j})\!)_{i, j}$:
\begin{align}
\label{ctrl_bij}
\abs{b_{i j}} & \leq 2^{\abs{i - j} - 1} (\Phi_{\setCa{h}} (\epsilon) \epsilon)^{\abs{i - j}} && \forall i\neq j ;\\
\label{ctrl_bii}
\abs{b_{i i} - 1} & \leq 2 \Phi_{\setCa{i}} (\epsilon) \epsilon^2 && \forall i
,\end{align}
for $\Phi_{\ref{i}}$ and $\Phi_{\ref{h}}$ some other absolute \quotes{quasi-one functions}.
\end{lem}

We apply Lemma~\ref{lem:alm-diag-mat} to the covariance matrix of $\vec{\Gamma}_{\integers{n}}$ (assuming $\epsilon$ was chosen small enough so that $\Phi_{\ref{g}}$ is finite); then, the formula for the density of Gaussian vectors gives that:
\begin{multline}
\frac{\dx{\Pr} (\vec{\Gamma}_{\integers{n}} = \vec{\gamma}_{\integers{n}})}{\dx{\gamma}}
\leq \biggl(\frac{\exp (\Phi_{\ref{g}} (\epsilon) \epsilon^2)}{2\pi \sigma^2}\biggr)^{n \div 2} \times \\
\exp \biggl(\frac{1}{2 \sigma^2} \Bigl((-1 + 2 \Phi_{\ref{i}} (\epsilon) \epsilon^2) \sum_{i \in \integers{n}} \gamma_i^2 + {\textstyle\frac{1}{2}} \sum_{\substack{i, j \in \integers{n}\\ i \neq j}} \abs{\gamma_i} \abs{\gamma_j} \times (2 \Phi_{\ref{h}} (\epsilon) \epsilon)^{\abs{i - j}}\Bigr)\biggr)
.\end{multline}
Bounding above $\abs{\gamma_i} \abs{\gamma_j}$ by $\frac{1}{2} (\gamma_i^2 + \gamma_j^2)$, that is bounded again by
\[\label{bound-dPg/dg}
\bigl(\Phi_{\setCa{j}} (\epsilon) \div 2 \pi \sigma^2\bigr)^{n \div 2} \exp \Bigl(-\frac{1}{2\Phi_{\setCa{k}}(\epsilon) \sigma^2} \sum_{i \in \integers{n}} \gamma_i^2\Bigr)
,\]
where
\[
\Phi_{\ref{j}}(\epsilon) \coloneqq \exp (\Phi_{\ref{g}}(\epsilon) \epsilon^2)
\]
and
\[
\Phi_{\ref{k}} (\epsilon) \coloneqq \Bigl(1 - 2 \Phi_{\ref{i}} (\epsilon) \epsilon^2 - {\textstyle \frac{1}{2}} \sum_{z \in \ZZ^*} (2 \Phi_{\ref{h}} (\epsilon) \epsilon)^{\abs{z}} \Bigr)_+^{-1}
\]
are \quotes{quasi-one functions} again.

In the sequel we assume that $\epsilon$ was chosen small enough so that $\Phi_{\ref{j}}(\epsilon),\ab \Phi_{\ref{k}}(\epsilon) < \infty$; and we define the following vectorial random variable (which we are actually only interested in through its law):

\begin{dfn}[Variable $\vec{\Pi}_{\NN}$]
$\vec{\Pi}_{\NN}$ is a random vector on $\RR^{\NN}$ whose entries are i.i.d.\ $\mcal{N} (0,\ab \Phi_{\ref{k}} (\epsilon) \sigma^2)$.
\end{dfn}

Then, Equation~\eqref{bound-dPg/dg} can be rephrased into:
\[
\frac{\dx{\Pr} (\vec{\Gamma}_{\integers{n}} = \vec{\gamma}_{\integers{n}})} {\dx{\Pr} (\vec{\Pi}_{\integers{n}} = \vec{\gamma}_{\integers{n}})} \leq ( \Phi_{\ref{j}}(\epsilon) \Phi_{\ref{k}}(\epsilon) )^{n \div 2} \quad \text{uniformly in $\vec{\gamma}_{\integers{n}}$}
.\]

Therefore, for $I \subset \integers{n}$ with $\abs{I} \geq p n$:
\begin{multline}\label{prefinal bound}
\Pr \bigl(\forall i \in I \quad \Gamma_i \geq \alpha H^{-1/2} (\log i)_+^{1/2}\bigr) \leq \\
(\Phi_{\ref{j}} (\epsilon) \Phi_{\ref{k}} (\epsilon))^{n \div 2} \Pr \bigl(\forall i \in I \quad \Pi_i \geq \alpha H^{-1/2} (\log i)_+^{1/2}\bigr)
.\end{multline}
But
\begin{multline}\label{final bound}
\Pr \bigl(\forall i \in I \quad \Pi_i \geq \alpha H^{-1/2} (\log i)_+^{1/2}\bigr) \\
= \prod_{i \in I} \Pr \bigl(\Pi_i \geq \alpha H^{-1/2} (\log i)_+^{1/2}\bigr) \\
= \prod_{i \in I} \Pr\bigl(\mcal{N} (0, 1) \geq C_{\setCa{l}} (\log i)_+^{1/2}\bigr) \\
\leq \prod_{i \in I} \exp \bigl(-C_{\ref{l}}^2 (\log i)_+ \div 2\bigr) \\
= \prod_{i\in I} (i\vee1)^{-C_{\ref{l}}^2 \div 2} \\
\leq (\abs{I} - 1)_+!^{-C_{\ref{l}}^2 \div 2}
\leq (\ceil{p n} - 1)_+!^{-C_{\ref{l}}^2 \div 2}
\end{multline}
(with $C_{\ref{l}} \coloneqq \alpha H^{-1/2} \div \Phi_{\ref{k}} (\epsilon)^{1/2} \sigma$), where the penultimate inequality comes from ordering $I \eqqcolon \{i_0, i_1, \ldots, i_{\abs{I} -1}\}$ with $i_0 < i_1 < \cdots$, and observing that then $i_j \geq j$ for all $j$, so that $(i_j \vee 1)^{-C_{\ref{l}}^2 \div 2} \leq (j \vee 1)^{-C_{\ref{l}}^2 \div 2}$.

Combining \eqref{prefinal bound} with \eqref{final bound} shows that $\Pr \bigl(\Gamma_i \geq \alpha H^{-1/2} (\log i)_+^{1/2}\ \forall i \in I\bigr)$ decreases superexponentially in $n$ uniformly in $I$, which finally proves Proposition~\ref{ppn:P(An') superexponential} and hence Theorem~\ref{thm:main}.\qed

\appendix

\section{Conditional expectation of the fBm}\label{sec:Kuv}

This appendix is devoted to ending the proof of Lemma~\ref{lem:Drift+Levy} initiated in \S~\ref{sec:Drift+Levy}. At the point we have got to, what remains to do is showing that
\[\label{Drift(W)}
C_1 \int_{s = -\infty}^0 \bigl((v - s)^{\Hmoh} - (-s)^{\Hmoh}\bigr) \dx{W}_{\tau + s}
\]
(seen as a trajectory indexed by $v \in \RR_+$) is actually equal to $\Drift((Z_{\tau + u} - Z_\tau)_{u \leq 0})$ with $\Drift$ defined by \eqref{DriftAsIntegral.1}-\eqref{Kuv}, where $W$ is the ordinary Brownian motion driving the fBm $Z$. To alleviate notation, actually we will only prove this result for $\tau \equiv 0$, the original case being the same up to time translation of the increments (hence the informal definition \eqref{Drift} of $\Drift$).

The starting point for our computation is the \emph{Pipiras–Taqqu formula}, which says that Equation~\eqref{ZofdW} defining the past increments of $Z$ as a function of the past increments of $W$ has an \quotes{inverse} giving back the past increments of $W$ from the past increments of $Z$:

\begin{ppn}[{\cite[Cor.~1.1]{Pipiras-Taqqu}}]
In the context of \S~\ref{sec:Context}, one has almost-surely, for all $t$:
\[\label{WofdZ}
W_t = \frac{C_1^{-1}}{\Pi (\Hmoh) \Pi (-\Hmoh)} \int_{\RR} \bigl((t - s)_+^{-\Hmoh} - (-s)_+^{-\Hmoh}\bigr) \dx{Z}_s
.\]
(Recall that $\Pi (\bcdot)$ is Euler's pi function extrapolating the factorial).
\end{ppn}

\begin{NB}
From now on in this appendix, it will be convenient to shorthand \quotes{$1 \div \Pi(\Hmoh) \Pi(-\Hmoh)$} into \quotes{$C_H$}.
\end{NB}

So, let us use \eqref{WofdZ} to get \eqref{DriftAsIntegral.1}-\eqref{Kuv}. First, like \eqref{ZofdW}, Equations~\eqref{Drift(W)} and~\eqref{WofdZ} have to be interpreted by integrating by parts: for $v \geq 0, s \leq 0$, that means resp.\ that:
\[
\eqref{Drift(W)} = \Hmoh C_1 \int_{-\infty}^0 \xi_{\Hmoh - 1} (-s, v) W_s \dx{s}
;\]
\begin{multline}\label{WofZ}
\frac{W_t}{C_1^{-1} C_H} = \\
\Hmoh \int_{-\infty}^t \xi_{-\Hmoh - 1} (t - s, -t) (Z_s - Z_t) \dx{s} + \Hmoh \int_t^0 (-s)^{-\Hmoh-1} Z_s \dx{s} + (-t)^{-\Hmoh} Z_t
,\end{multline}
where we recall that $\xi_r (a, b) \coloneqq (a + b)^r - a^r$. Hence, \eqref{Drift(W)} is equal to:
\begin{align}
\label{1479a}
& \mathbin{\phantom{+}} \Hmoh^2 C_H \int_{s = -\infty}^0 \int_{u = -\infty}^s \xi_{\Hmoh - 1} (-s, v) \xi_{-\Hmoh - 1} (s - u, -s) (Z_u - Z_s) \dx{s} \dx{u} \\
\label{1479b}
& + \Hmoh^2 C_H \int_{s = -\infty}^0 \int_{u = s}^0 \xi_{\Hmoh - 1} (-s, v) (-u)^{-\Hmoh - 1} Z_u \dx{s} \dx{u} \\
\label{1479c}
& + \Hmoh C_H \int_{-\infty}^0 \xi_{\Hmoh - 1} (-s, v) (-s)^{-\Hmoh} Z_s \dx{s} .
\end{align}

Now we are going to re-write each of the terms \eqref{1479a}–\eqref{1479c} as an integral against $Z_u \dx{u}$, in order to get \eqref{DriftAsIntegral.1}-\eqref{Kuv}. First, Term~\eqref{1479c} is already of the wanted form, up to renaming \squotes{$s$} into \squotes{$u$}. Next, Term~\eqref{1479b} simplifies into:
\begin{multline}\label{1354}
\eqref{1479b} =
\Hmoh^2 C_H \int_{u = -\infty}^0 \Bigl(\int_{s = -\infty}^u \xi_{\Hmoh - 1} (-s, v) \dx{s}\Bigr) (-u)^{-\Hmoh - 1} Z_u \dx{u} \\
= -\Hmoh C_H \int_{-\infty}^0 \Bigl[ \xi_{\Hmoh}(-s,v) \Bigr]_{s=-\infty}^u (-u)^{-\Hmoh-1} Z_u \dx{u} \\
= -\Hmoh C_H \int_{-\infty}^0 \xi_{\Hmoh}(-u, v) (-u)^{-\Hmoh - 1} Z_u \dx{u} .
\end{multline}

Term \eqref{1479a} is the hardest to get into the wanted form, because splitting naively the factor $(Z_u-Z_s)$ would yield divergent integrals. To bypass that problem, we first make a truncation: for $\epsilon$ a small positive number,
\begin{align}
\eqref{1479a} & \approx \Hmoh^2 C_H \int_{s = -\infty}^0 \int_{u = -\infty}^{(1 + \epsilon) s} \xi_{\Hmoh - 1}(-s, v) \xi_{-\Hmoh - 1} (s - u, -s) (Z_u - Z_s) \dx{s} \dx{u} \nonumber\\
\label{7204a}
& = \mathbin{\phantom{+}} \Hmoh^2 C_H \iint_{\substack{s < 0\\ u < (1 + \epsilon) s}} \xi_{\Hmoh - 1} (-s, v) \xi_{-\Hmoh - 1} (s - u, -s) Z_u \dx{s} \dx{u} \\
\label{7204b}
& \mathrel{\phantom{=}} - \Hmoh^2 C_H \iint_{\substack{s < 0\\ u < (1 + \epsilon) s}} \xi_{\Hmoh - 1} (-s, v) \xi_{-\Hmoh - 1} (s - u, -s) Z_s \dx{s} \dx{u}
.\end{align}
By the change of variables $(s, u) \leftarrow (u - s, u)$,
\[
\eqref{7204a} = \Hmoh^2 C_H \int_{u = -\infty}^0 \Bigl(\int_{s = -\infty}^{\epsilon u \div (1 + \epsilon)} \1{s > u} \xi_{\Hmoh - 1} (s - u, v) \xi_{-\Hmoh - 1} (-s, s - u) \dx{s}\Bigr) Z_u \dx{u}
;\]
and by the change of variables $(s, u) \leftarrow (u - s, s)$,
\begin{multline}
\eqref{7204b}
= -\Hmoh^2 C_H \int_{u = -\infty}^0 \Bigl(\int_{s = -\infty}^{\epsilon u} \xi_{\Hmoh - 1} (-u, v) \xi_{-\Hmoh - 1} (-s, -u) \dx{s}\Bigr) Z_u \dx{u} \\
\approx -\Hmoh^2 C_H \int_{u = -\infty}^0 \Bigl(\int_{s = -\infty}^{\epsilon u \div (1 + \epsilon)} \xi_{\Hmoh - 1} (-u, v) \xi_{-\Hmoh - 1} (-s, -u) \dx{s}\Bigr) Z_u \dx{u}
,\end{multline}
where by \quotes{$\approx$} we mean that, for all $v$, the difference between the two members from either side of the \squotes{$\approx$} sign tends to $0$ as $\epsilon \to 0$, as one can check by simple estimates.
So,
\[\label{4625}
\eqref{1479a} \approx \Hmoh^2 C_H \int_{-\infty}^0 \Bigl(\int_{-\infty}^{\epsilon u \div (1 + \epsilon)} J (v, u, s) \dx{s}\Bigr) Z_u \dx{u}
,\]
with
\[
J (v, u, s) \coloneqq \1{s > u} \xi_{\Hmoh - 1} (s - u, v) \xi_{-\Hmoh - 1} (-s, s - u) - \xi_{\Hmoh - 1} (-u, v) \xi_{-\Hmoh - 1}(-s, -u)
.\]
But $\int^0 J (v, u, s) \dx{s}$ does converge, so, letting $\epsilon$ tend to $0$, we get in the end:
\[\label{4625-0}
\eqref{1479a} = \Hmoh^2 C_H \int_{-\infty}^0 \Bigl(\int_{-\infty}^0 J (v, u, s) \dx{s}\Bigr) Z_u \dx{u}
.\]

Summing \eqref{1479c}, \eqref{1354} and \eqref{4625-0}, and observing that $\xi_{\Hmoh - 1} (-u, v) (-u)^{-\Hmoh} -\ab \xi_{\Hmoh} (-u, v) (-u)^{-\Hmoh - 1} = -v (v - u)^{\Hmoh - 1} (-u)^{-\Hmoh - 1}$, finally yields Equation~\hbox{\eqref{DriftAsIntegral.1}-\eqref{Kuv}}.\qed

\section{On thick subsets of $\NN$}\label{sec:thick}

Remember that we have defined a subset $\mcal{I} \subset \NN$ to be \emph{thick} when
\[
\limsup_{n \to \infty} \frac{\abs{\mcal{I} \cap \integers{n}}}{n} > 0
,\]
where $\integers{n} \coloneqq \{0, \ldots, n - 1\}$ is the set of the $n$ smallest natural integers. In this appendix, we will prove two basic properties of thick subsets of $\NN$ which we used in the body of the article:
\begin{lem}\label{lem:2179}
Let $\mcal{I}$ be a thick set of integers and let $k > 1$. Then there exists $l \in \integers{k}$ such that the set
\[\label{J}
\{j \in \NN|\ j k + l \in \mcal{I}\}
\]
is thick.
\end{lem}

\begin{proof}
Let $\mcal{I}$ be a subset of $\NN$; and for $l \in \integers{k}$, denote the set defined by \eqref{J} as \squotes{$\mcal{J}_l$}. Then the $\mcal{J}_l$'s make a partition of $\mcal{I}$; so, for $n \in \NN$:
\[
\mcal{I} \cap \integers{n} = \bigsqcup_{l \in \integers{k}} \bigl((k \mcal{J}_l + l) \cap \integers{n}\bigr) = \bigsqcup_{l \in \integers{k}} \bigl(k \bigl(\mcal{J}_l \cap \integers{\ceil{(n - l) \div k}}\bigr) + l\bigr)
.\]
Thus,
\begin{multline}
\frac{\abs{\mcal{I} \cap \integers{n}}}{n}
= \sum_{l \in \integers{k}} \frac{\abs{\mcal{J}_l \cap \integers{\ceil{(n - l) \div k}}}}{n} \\
= \sum_{l \in \integers{k}} \frac{\ceil{(n - l) \div k}}{n} \frac{\abs{\mcal{J}_l \cap \integers{\ceil{(n - l) \div k}}}}{\ceil{(n - l) \div k}}
;\end{multline}
and therefore, letting $n \to \infty$:
\[
\limsup_{n \to \infty} \frac{\abs{\mcal{I} \cap \integers{n}}}{n} \leq \sum_{l \in \integers{k}} k^{-1} \limsup_{m \to \infty} \frac{\abs{\mcal{J}_l \cap \integers{m}}}{m}
.\]
So, for $\mcal{I}$ to be thick, one at least of the $\mcal{J}_l$'s has to be thick, proving the lemma.
\end{proof}

\begin{lem}\label{lem:harmonicthick}
If $\mcal{I}\subset\NN$ is thick, then the series
\[
\sum_{i\in\mcal{I}} i^{-1}
\]
is divergent.
\end{lem}

\begin{proof}
Assume that $\mcal{I}$ is thick and let
\[
p \coloneqq \limsup_{n \to \infty} \frac{\abs{\mcal{I} \cap \integers{n}}}{n}
\]
(which then is positive); and fix arbitrary $p > p' > p'' > 0$.

Then, there will be arbitrarily large $n \in \NN$ such that%
\[\label{high-density n}
\abs{\mcal{I} \cap \integers{n}} \geq p' n
,\]
so that we can define the following sequence by induction: $n_0 = 1$, and $n_{k + 1}$ is the smallest $n$ satisfying \eqref{high-density n} such that $n_{k + 1} \geq n_k \div (p' - p'')$.

Now, observe that
\[
\abs{\mcal{I} \cap [n_k, n_{k + 1})}
\geq \abs{\mcal{I} \cap \integers{n_{k + 1}}} - n_k
\geq p' n_{k + 1} - (p' - p'') n_{k + 1}
= p'' n_{k + 1},
\]
so that
\[
\sum_{\mcal{I} \cap [n_k, n_{k + 1})} i^{-1}
\geq \abs{\mcal{I} \cap [n_k, n_{k + 1})} \inf_{i \in \mcal{I} \cap [n_k, n_{k + 1})} i^{-1}
\geq p'' n_{k + 1} \times n_{k + 1}^{-1} = p''
,\]
and therefore:
\[
\sum_{\substack{i \in \mcal{I}\\ i \geq 1}} i^{-1}
\geq \sum_{k = 0}^{\infty} \sum_{\substack{i \in \mcal{I}\\ n_k \leq i < n_{k + 1}}} i^{-1}
\geq \sum_{k = 0}^{\infty} p''
= \infty
,\]
proving the lemma.
\end{proof}

\section{Explicit estimate for the supremum of Gaussian processes}\label{sec:supGauss}

\begin{proof}[\proofname\ of Lemma~\ref{lem:regGaussian}]
Let $\theta \in (0, 1]$ and let $X$ satisfying the assumptions of the lemma. Then obviously, for the continuous version of $X$:
\[
\norm{X} (\omega) \leq \sum_{i = 0}^\infty \sup_{a \in \integers{2^i}} \abs{X (\omega)_{a 2^{-i}} - X (\omega)_{(a + 1) 2^{-i}}}
.\]
Therefore, for $(\gamma_i)_i$ a sequence of positive real numbers such that $\sum_i \gamma_i = 1$, one has that, for all $x \geq 0$:
\begin{multline}
\Pr (\norm{X} \geq x)
\leq \sum_{i = 0}^\infty \Pr (\sup_{a \in \integers{2^i}} \abs{X_{a 2^{-i}} - X_{(a + 1) 2^{-i}}} \geq \gamma_i x) \\
\leq \sum_i 2^i \sup_a \Pr (\abs{X_{a 2^{-i}} - X_{(a + 1) 2^{-i}}} \geq \gamma_i x)
.\end{multline}
But, uniformly in $a$,
\begin{multline}
\Pr (\abs{X_{a 2^{-i}} - X_{(a + 1) 2^{-i}}} \geq \gamma_i x)
= \Pr \bigl(\mcal{N} (0, 1) \geq \gamma_i x \div \var^{1/2} (X_{a 2^{-i}} - X_{(a + 1) 2^{-i}})\bigr) \\
\leq \Pr \bigl(\mcal{N} (0, 1) \geq 2^{i \theta} \gamma_i x\bigr)
\leq 2 \exp \bigl(-2^{2 i \theta - 1} (\gamma_i x)^2\bigr)
;\end{multline}
so, taking $\gamma_i \coloneqq (1 - 2^{-\theta \div 2}) 2^{-i \theta \div 2}$:
\[
\Pr (\norm{X} \geq x) \leq \sum_{i = 0}^{\infty} 2^{i + 1} \exp \bigl(-(1 - 2^{-\theta \div 2})^2 \times 2^{i \theta - 1} x^2\bigr)
.\]
Provided $x \geq 2 \div (1 - 2^{-\theta \div 2}) \theta^{1/2} \eqqcolon C_{\setCa{o}} (\theta)$, one has (bounding $2^{i \theta}$ below by $(1 + i \theta \log 2)$)
\[
\exp \bigl(-2^{i \theta} x^2 \div 2 (1 - 2^{-\theta \div 2})^2\bigr) \leq \exp (-x^2 \div 2 (1 - 2^{-\theta \div 2})^2 - \log (4) i\bigr)
,\]
so that, for $x \geq C_{\ref{o}}$:
\[
\Pr (\norm{X} \geq x) \leq \Bigl(\sum_{i = 0}^{\infty} 2^{i + 1} 4^{-i}\Bigr) \exp \bigl(-x^2 \div 2 (1 - 2^{-\theta \div 2})^2\bigr)
\eqqcolon 4 e^{-C_{\ref{c}} (\theta) x^2}
.\]
On the other hand, for $x < C_{\ref{o}}$ one has obviously $\Pr (\norm{X} \geq x) \leq 1$; so Equation~\eqref{regGaussian} follows with $C_{\ref{d}} \coloneqq 4 \vee e^{C_{\ref{c}} C_{\ref{o}}^2}$.
\end{proof}

\section{Almost diagonal matrices}\label{sec:alm-diag-mat}

\begin{proof}[\proofname\ of Lemma~\ref{lem:alm-diag-mat}]
Consider $\mbf{A}$ satisfying the assumptions of the lemma, and denote $\mbf{I}_n - \mbf{A} \eqqcolon \mbf{H}$.
The first part of this proof will consist in deriving estimates on the entries of $\mbf{H}$ and its powers. Denote respectively
\begin{align}
\mbf{H} & \eqqcolon (\!(h_{i j})\!)_{i, j \in \integers{n}} ;\\
\mbf{H}^k & \eqqcolon (\!(h^{(k)}_{i j})\!)_{i, j \in \integers{n}} && \forall k \geq 0 .
\end{align}
Then the assumptions of the lemma ensure that one has
$\abs{h_{i i}} = 0\ \forall i$, resp.\ab\ $\abs{h_{i j}} \leq \epsilon^{\abs{i - j}}\ab\ \forall i \neq j$, and hence
\[
\abs{h^{(k + 1)}_{i j}} \leq \sum_{i' \in \integers{n}} \abs{h_{i i'}} \abs{h^{(k)}_{i' j}} = \sum_{i' \neq i} \epsilon^{\abs{i - i'}} \abs{h^{(k)}_{i' j}} \quad \forall i, j \quad \forall k
.\]
That suggests to define by induction:
\[
\left\{\begin{aligned}
\mfrak{h}^{(0)}_z & \coloneqq \1{z = 0} && \forall z \in \ZZ \\
\mfrak{h}^{(k + 1)}_z & \coloneqq \sum_{z' \neq z} \epsilon^{\abs{z - z'}} \mfrak{h}^{(k)}_{z'} && \forall z \in \ZZ \quad \forall k \geq 0 ,
\end{aligned}\right.
\]
so that one has
\[\label{h leq frak-h}
\abs{h^{(k)}_{i j}} \leq \mfrak{h}^{(k)}_{i - j} \quad \forall i, j \quad \forall k
.\]

The interest of having introduced the $\mfrak{h}^{(k)}_z$'s is that these are easier to bound than the $h^{(k)}_{i j}$'s themselves. To bound the $\mfrak{h}^{(k)}_z$'s, we begin with observing that one has obviously by induction:
\begin{multline}\label{bound-hkz_naive}
\mfrak{h}^{(k)}_{z}
= \sum_{\substack{(0, s_1, s_2, \ldots, s_{k - 1}, z)\\ 0 \neq s_1, s_1 \neq s_2, \ldots, s_{k - 1} \neq z}} \epsilon^{\sum_{i \in \integers{k}}\abs{s_i - s_{i + 1}}} \\
= \sum_{n \geq 0} \card \Bigl\{(0, s_1, s_2, \ldots, s_{k - 1}, z)\Big|\ s_{i + 1} \neq s_i\ \forall i\ \text{and}\ \sum_i \abs{s_i - s_{i + 1}} = n\Bigr\} \epsilon^{n}
.\end{multline}
To bound the cardinality appearing in \eqref{bound-hkz_naive}, we observe that a $(k + 1)$-tuple $(0,\ab s_1,\ab s_2,\ab \ldots,\ab s_{k - 1},\ab z)$ such that $s_{i + 1} \neq s_i\ \forall i$ and $\sum_i \abs{s_i - s_{i + 1}} = n$ (we will call such a $(k + 1)$-tuple as \emph{valid}) can be \emph{coded} by a word of $n$ symbols from $\{\psymb, \Psymb, \msymb, \Msymb\}$, in the following way: successively, for each $i$ we write $(\abs{s_{i + 1} - s_{i}} - 1)$ symbols \quotes{$\sign (s_{i + 1} - s_i)$} followed by a symbol \quotes{$\sign (s_{i + 1} - s_i)_|$}—for instance, for $k = 5$, $z = 2$, $n = 8$, one would have
\[
(0, 1, 4, 2, 1, 2) \mapsto \text{\quotes{$\Psymb\psymb\psymb\Psymb\msymb\Msymb\Msymb\Psymb$}}
.\]
Obviously such a coding in injective. Moreover, for given $k, z, n$, an $n$-character word may be the image of a valid $(k + 1)$-tuple only if $z$ and $n$ have the same parity, that the word contains $(n + z) \div 2$ symbols from $\{\psymb, \Psymb\}$ vs.\ $(n - z) \div 2$ symbols from $\{\msymb, \Msymb\}$, and that $k$ exactly of the $n$ symbols, necessarily including the last one, are from $\{\Psymb, \Msymb\}$. Henceforth:
\begin{multline}\label{bound on card}
\card \Bigl\{(0, s_1, s_2, \ldots, s_{k - 1}, z)\Big|\ s_{i + 1} \neq s_i\ \forall i\ \text{and}\ \sum_i \abs{s_i - s_{i + 1}} = n\Bigr\} \\
\leq \1{2 \mid n - z} \binom{n}{(n - \abs{z}) \div 2} \binom{n - 1}{k - 1}
.\end{multline}
In the end, combining \eqref{h leq frak-h}, \eqref{bound-hkz_naive} and \eqref{bound on card}:
\[\label{thebound_hkij}
\abs{h^{(k)}_{i j}} \leq \sum_{m \geq 0} \binom{\abs{i - j} + 2 m}{m} \binom{\abs{i - j} + 2 m - 1}{k - 1} \epsilon^{\abs{i - j} + 2 m} \quad \forall i, j \quad \forall k
.\]

\bigskip

After these preliminary estimates, let us turn to proving the lemma itself. We begin with the first part, namely, bounding $\det \mbf{A}$ below. For $\mbf{X} \eqqcolon (\!(x_{i j})\!)_{i, j}$ an $n \times n$ matrix, denote
\[
\norm{\mbf{X}} \coloneqq \sup_{j \in \integers{n}} \sum_{i \in \integers{n}} \abs{x_{i j}}
:\]
$\norm{\bcdot}$ is the operator norm of $\mbf{X}$ when seen as an operator from $\ell^1 (\integers{n})$ into itself, so it is sub-multiplicative. Then, the formula
\[
\log (\mbf{I}_n - \mbf{H}) = \sum_{k = 1}^{\infty} k^{-1} \mbf{H}^{k}
\]
converges as soon as $\norm{\mbf{H}} < 1$, then yielding:
\begin{multline}
\abs{\trace \log (\mbf{I}_n - \mbf{H})}
\leq \sum_{k = 1}^{\infty} k^{-1} \abs{\trace \mbf{H}^k} \\
\leq 0 + \tsfrac{1}{2} \abs{\trace \mbf{H}^2} + n \sum_{k \geq 3} k^{-1} \norm{\mbf{H}^k} \\
\leq \tsfrac{1}{2} \abs{\trace \mbf{H}^2} + n \sum_{k \geq 3} k^{-1} \norm{\mbf{H}}^k
.\end{multline}
But the assumptions on the entries of $\mbf{H}$ imply that
\[\label{norm-H}
\norm{\mbf{H}} \leq \sum_{z \in \mbf{Z}^*} \epsilon^{\abs{z}} = \frac{2 \epsilon}{1 - \epsilon}
,\]
which is $< 1$ as soon as $\epsilon < 1/3$; and on the other hand, we get from \eqref{thebound_hkij} that, for all $i$,
\[\label{bound_h2ii}
\abs{h^{(2)}_{i i}}
\leq \sum_{m \geq 1} \binom{2 m}{m} (2 m - 1) \epsilon^{2 m} \leq 2 \epsilon^2 + \sum_{m \geq 2} 2^{2 m} (2 m - 1) \epsilon^{2 m}
\eqqcolon 2 \Phi_{\setCa{m}} (\epsilon) \epsilon^2
,\]
so that $\abs{\trace \mbf{H}^2} \leq 2 n \Phi_{\ref{m}} (\epsilon) \epsilon^2$. In the end:
\begin{multline}
\det \mbf{A}
= \det \exp \log (\mbf{I}_n - \mbf{H})
= \exp \trace \log (\mbf{I}_n - \mbf{H}) \\
\geq \exp \Biggl(-n \Phi_{\ref{m}} (\epsilon) \epsilon^2 - n \sum_{k \geq 3} k^{-1} \biggl(\frac{2 \epsilon}{1 - \epsilon} \biggr)^{k}\Biggr) \eqqcolon \exp (-n \Phi_{\ref{g}} (\epsilon) \epsilon^2)
,\end{multline}
which is Equation~\eqref{det-alm-diag-mat}.

\bigskip

Now let us handle the second part of the lemma, namely, bounding the entries of $(\mbf{A}^{-1} - \mbf{I}_n)$. Provided $\norm{\mbf{H}} < 1$, one has
\[
\mbf{A}^{-1} - \mbf{I}_n = \sum_{k = 1}^{\infty} \mbf{H}^k
,\]
so that
\[\label{leq_sum_hk}
\abs{b_{i j} - \1{i = j}} \leq \sum_{k = 1}^{\infty} \abs{h^{(k)}_{i j}} \quad \forall i, j
.\]
To bound the r-h.s.\ of \eqref{leq_sum_hk}, we write that, starting from \eqref{thebound_hkij}:
\begin{multline}\label{computn-toto}
\sum_{k \geq 1} \abs{h^{(k)}_{i j}}
\leq \sum_{m \geq 0} \binom{\abs{i - j} + 2 m}{m} \biggl(\sum_{k = 1}^{\abs{i - j} + 2 m} \binom{\abs{i - j} + 2 m - 1}{k - 1}\biggr) \epsilon^{\abs{i - j} + 2 m} \\
= \sum_{m \geq 0} \binom{\abs{i - j} + 2 m}{m} \1{\abs{i - j} + 2 m \geq 1} 2^{\abs{i - j} + 2 m - 1} \epsilon^{\abs{i - j} + 2 m} \\
= \biggl(\sum_{m \geq 0} \1{\abs{i - j} + 2 m \geq 1} \binom{\abs{i - j} + 2 m}{m} (4 \epsilon^2)^{m}\biggr) \times 2^{\abs{i - j} - 1} \epsilon^{\abs{i - j}}
.\end{multline}
But we observe that, for $z \geq 1$, $x > 0$,
\begin{multline}
\sum_{m \geq 0} \binom{z + 2 m}{m} x^m
= 1 + \sum_{m \geq 1} \binom{z + 2 m}{m} x^m
\leq 1 + \sum_{m \geq 1} \frac{(z + 2 m)^{m}}{m!} x^m \\
\leq 1 + \sum_{m \geq 1} \frac{(2 z)^m + (4 m)^m}{m!} x^m \\
\leq 1 + \sum_{m \geq 1} \frac{(2 z x)^m}{m!} + \sum_{m \geq 1} (4e x)^m
= e^{2 z x} + \frac{4e x}{1 - 4e x} \\
\leq \biggl(e^{2 x} + \frac{4e x}{1 - 4e x}\biggr)^z
\eqqcolon \Phi_{\setCa{n}} (x)^z
,\end{multline}
so that here, for $i \neq j$:
\[
\sum_{k = 1}^{\infty} \abs{h^{(k)}_{i, j}} \leq \Phi_{\ref{n}} (4 \epsilon^2)^{\abs{i - j}} \times 2^{\abs{i - j} - 1} \epsilon^{\abs{i - j}} \eqqcolon 2^{\abs{i - j} - 1} (\Phi_{\ref{h}} (\epsilon) \epsilon)^{\abs{i - j}}
,\]
which is Equation~\eqref{ctrl_bij}.

Equation \eqref{ctrl_bii} for the case $i = j$ is derived in the same way as \eqref{ctrl_bij}, with just a few minor differences at the beginning of the computation: namely, in the l-h.s.\ of \eqref{computn-toto}, we treat apart the cases \quotes{$k = 1$} (which yields zero here since $h_{i i} = 0$ by assumption) and \quotes{$k = 2$} (which has already been handled by \eqref{bound_h2ii}); then all the sequel is the same.
\end{proof}

\begin{rmk}
The bounds \eqref{det-alm-diag-mat}–\eqref{ctrl_bii} of Lemma~\ref{lem:alm-diag-mat} are optimal at first order. Actually this is much more than needed to prove Theorem~\ref{thm:main}, and we could have got a sufficient result with a shorter proof; but it seemed interesting to me to state the sharp version of the lemma.
\end{rmk}

\bibliographystyle{plain}
\bibliography{bibliography}

\begin{thebibliography}{1}

\bibitem{Bender2012}
Christian Bender.
\newblock Simple arbitrage.
\newblock {\em Ann. Appl. Probab.}, 22(5):2067--2085, 2012.

\bibitem{Nourdin-book-fBm}
Ivan Nourdin.
\newblock {\em Selected aspects of fractional {B}rownian motion}, volume~4 of
  {\em Bocconi \& Springer Series}.
\newblock Springer, Milan; Bocconi University Press, Milan, 2012.

\bibitem{Pipiras-Taqqu}
Vladas Pipiras and Murad~S. Taqqu.
\newblock Deconvolution of fractional {B}rownian motion.
\newblock {\em J. Time Ser. Anal.}, 23(4):487--501, 2002.

\end{thebibliography}
\end{document}